\documentclass[preprint,12pt]{elsarticle}

\usepackage{amssymb}
\usepackage{amsthm}
\usepackage{graphics}
\usepackage{epsfig}
\usepackage{amssymb}
\usepackage{graphicx}
\usepackage{subfigure}
\usepackage{color}
\usepackage{tikz}
\usetikzlibrary{patterns}
\usepackage{hyperref}
\usepackage{a4wide}
\usepackage{amsmath}
\usepackage{amsfonts}
\usepackage{enumerate}

\newtheorem*{theoremaux}{Theorem \theoremauxnum}
\gdef\theoremauxnum{1}

\newtheorem{lemma}{\bf Lemma}[section]
 
\newtheorem{theorem}{\bf Theorem}[section]
\newtheorem{proposition}[lemma]{\bf Proposition}

\newtheorem{definition}{\bf Definition}[section]

\journal{~}

\begin{document}

\begin{frontmatter}



\title{On Perfectness of Annihilating-Ideal Graph of $\mathbb{Z}_n$}



\author{Manideepa Saha}
\ead{manideepasaha1991@gmail.com }
\author{Sucharita Biswas}
\ead{biswas.sucharita56@gmail.com}
\author{Angsuman Das\corref{cor1}}
\ead{angsuman.maths@presiuniv.ac.in}
\address{Department of Mathematics, Presidency University, Kolkata, India}

\cortext[cor1]{Corresponding author}

\begin{abstract}
The annihilating-ideal graph of a commutative ring $R$ with unity is defined as the graph $\mathbb{AG}(R)$ with the vertex set is the set of all non-zero ideals with non-zero annihilators and two distinct vertices $I$ and $J$ are adjacent if and only if $IJ = 0$. Nikandish {\it et.al.} proved that $\mathbb{AG}(\mathbb{Z}_n)$ is weakly perfect. In this short paper, we characterize $n$ for which $\mathbb{AG}(\mathbb{Z}_n)$ is perfect.
\end{abstract}


\end{frontmatter}

\section{Introduction}
Over the last two decades, various graphs defined on rings has become an interesting topic of research. Various graphs like \cite{angsu-armendariz},\cite{anderson-livingston},\cite{badawi1},\cite{badawi2},\cite{badawi3},\cite{1st-paper},\cite{angsu-bedanta},\cite{mks-ideal} have been constructed to study the interplay between the graph theoretic and ring-theoretic properties. Interested readers are referred to the following surveys \cite{survey1},\cite{survey2} on graphs defined on rings. One such graph is the annihilating-ideal graph $\mathbb{AG}(R)$ of a commutative ring $R$, introduced by \cite{1st-paper}.

\begin{definition}\cite{1st-paper} Let $R$ be a commutative ring with unity. The annihilating-ideal graph of $R$ is defined as the graph $\mathbb{AG}(R)$ with the vertex set is the set of all non-zero ideals with non-zero annihilators and two distinct vertices $I$ and $J$ are adjacent if and only if $IJ = 0$.
\end{definition}

In \cite{weakly-perfect}, the authors proved that $\mathbb{AG}(\mathbb{Z}_n)$ is {\it weakly perfect}, i.e., its clique number $\omega$ is equal to its chromatic number $\chi$. 

A graph $G$ is said to be {\it perfect} if $\omega(H)=\chi(H)$ holds for all induced subgraphs $H$ of $G$. Perfect graphs play an important role in graph theory, as many hard graph problems in general like graph coloring, finding maximum clique and independent set, etc. can be solved in polynomial-time in case of perfect graphs. Thus characterizing perfect graphs in different families \cite{ebrahimi}, \cite{angsu-intersection-perfect} is an important issue. In this short paper, we characterize $n$ for which $\mathbb{AG}(\mathbb{Z}_n)$ is perfect. The following theorem is the main result of the paper:

\begin{theorem}\label{main-theorem}
$\mathbb{AG}(\mathbb{Z}_n)$ is perfect if and only if $n$ is one of the form $p^{\alpha_1}_1,~ p^{\alpha_1}_1p^{\alpha_2}_2,~p^{\alpha_1}_1p_2p_3$ or $p_1p_2p_3p_4$, where $p_i$'s are distinct primes and $\alpha_{i}\in \mathbb{N}$.
\end{theorem}

In the next section, we prove Theorem \ref{main-theorem}. Before that we state an observation and an important result which will be crucial in our proof.

\begin{proposition}\label{adj}
	The vertex set of $\mathbb{AG}(\mathbb{Z}_n)$ is $\{ \langle m \rangle : m\mid n, 1<m<n \}$ and two vertices $\langle m_1\rangle$ and $\langle m_2\rangle$ are adjacent if and only if $n \mid m_1m_2$. 
\end{proposition}

\begin{theorem}{(Strong Perfect Graph Theorem)}
	A graph $G$ is perfect if and only if neither $G$ nor $G^c$ has an induced odd-cycle of length greater or equal to $5$.
\end{theorem}

\section{Proof of Theorem \ref{main-theorem}}

We split the proof of Theorem \ref{main-theorem} into different cases (lemmas) depending upon the number of distinct prime factors of $n$.

First, we deal with the case when $n$ has more than $4$ distinct prime factors and show that in this case $\mathbb{AG}(\mathbb{Z}_n)$ is not perfect.

\begin{lemma}
If $n={p_1}^{\alpha_1}{p_2}^{\alpha_2}\cdots {p_k}^{\alpha_k}$ and $k\geq 5$, then $\mathbb{AG}(\mathbb{Z}_n)$ is not perfect.
\end{lemma}

\begin{proof}
Let $m=n/({p_1}^{\alpha_1}{p_2}^{\alpha_2}\cdots {p_5}^{\alpha_5})$. Then the following five vertices, taken in order, $$\langle {p_1}^{\alpha_1}{p_2}^{\alpha_2}{p_4}^{\alpha_4}m \rangle, \langle {p_3}^{\alpha_3}{p_4}^{\alpha_4}{p_5}^{\alpha_5}m \rangle, \langle {p_1}^{\alpha_1}{p_2}^{\alpha_2}{p_3}^{\alpha_3}m \rangle, \langle {p_2}^{\alpha_2}{p_4}^{\alpha_4}{p_5}^{\alpha_5}m \rangle, \langle {p_1}^{\alpha_1}{p_3}^{\alpha_3}{p_5}^{\alpha_5}m \rangle$$ form an induced $5$-cycle in $\mathbb{AG}(\mathbb{Z}_n)$. The adjacency and non-adjacency follows from Proposition \ref{adj}. Hence, by strong perfect graph theorem, the lemma follows.
\end{proof}

Next we focus on the case when $n$ has exactly $4$ distinct prime factors. We characterize the condition when $\mathbb{AG}(\mathbb{Z}_n)$ is perfect.

\begin{lemma}
If $n={p_1}^{\alpha_1}{p_2}^{\alpha_2}{p_3}^{\alpha_3}{p_4}^{\alpha_4}$, then $\mathbb{AG}(\mathbb{Z}_n)$ is perfect if and only if $\alpha_{i}=1$ for all $i$.	
\end{lemma}

\begin{proof}
Let at least one $\alpha_{i}>1$, say $\alpha_1>1$.	Then the following five vertices, taken in order, $$\langle {p_1}^{\alpha_1}{p_4}^{\alpha_4} \rangle, \langle {p_2}^{\alpha_2}{p_3}^{\alpha_3}{p_4}^{\alpha_4} \rangle, \langle {p_1}^{\alpha_1}{p_2}^{\alpha_2} \rangle, \langle p_1{p_3}^{\alpha_3}{p_4}^{\alpha_4}\rangle, \langle {p_1}^{\alpha_1-1}{p_2}^{\alpha_2}{p_3}^{\alpha_3} \rangle$$ form an induced $5$-cycle in $\mathbb{AG}(\mathbb{Z}_n)$. As earlier, the adjacency and non-adjacency follows from Proposition \ref{adj}. Hence, by strong perfect graph theorem, $\mathbb{AG}(\mathbb{Z}_n)$ is not perfect.

Now, we assume that $n=p_1p_2p_3p_4$. Then $\mathbb{AG}(\mathbb{Z}_n)$ has $14$ vertices:
$$\begin{array}{llc}
\mbox{1st type:} &\langle p_1 \rangle, \langle p_2 \rangle, \langle p_3 \rangle, \langle p_4 \rangle & 4 \mbox{ vertices of degree }1\\
\mbox{2nd type:} &\langle p_1p_2 \rangle, \langle p_2p_3 \rangle, \ldots, \langle p_3p_4 \rangle & 6 \mbox{ vertices of degree }3\\
\mbox{3rd type:} &\langle p_1p_2p_3 \rangle, \langle p_2p_3p_4 \rangle, \langle p_1p_3p_4 \rangle, \langle p_1p_2p_4 \rangle & 4 \mbox{ vertices of degree }7\\
\end{array}$$
If possible, let $\mathbb{AG}(\mathbb{Z}_n)$ has an induced odd cycle $C$ of length $t\geq 5$. Thus $C$ must have a vertex of second type. Without loss of generality, let $\langle p_1p_2 \rangle$ be a vertex in $C$. As $\langle p_1p_2 \rangle$ is adjacent to three vertices, namely $\langle p_3p_4 \rangle, \langle p_1p_3p_4 \rangle, \langle p_2p_3p_4 \rangle$, at least two of them, must lie on $C$.

{\bf Case 1:} $\langle p_1p_3p_4 \rangle \sim \langle p_1p_2 \rangle \sim \langle p_3p_4 \rangle$ be a part of $C$. Let $\langle x \rangle$ be the next vertex on $C$, i.e., $\langle p_1p_3p_4 \rangle \sim \langle p_1p_2 \rangle \sim \langle p_3p_4 \rangle \sim \langle x \rangle$. Then by the adjacency condition of the last two vertices, we get $p_1p_2\mid x$. But this imply that $\langle x \rangle \sim \langle p_1p_3p_4 \rangle$, i.e., we get a chord in $C$, a contradiction.

{\bf Case 2:} $\langle p_2p_3p_4 \rangle \sim \langle p_1p_2 \rangle \sim \langle p_3p_4 \rangle$ be a part of $C$. In this case also, proceeding similalrly, we get a contradiction.

{\bf Case 3:} $\langle p_1p_3p_4 \rangle \sim \langle p_1p_2 \rangle \sim \langle p_2p_3p_4 \rangle$ be a part of $C$. However, in this case, we get a chord of the form $\langle p_1p_3p_4 \rangle \sim \langle p_2p_3p_4 \rangle$ in $C$, a contradiction.

Thus $\mathbb{AG}(\mathbb{Z}_n)$ has no induced odd cycle $C$ of length $t\geq 5$.

Now, we consider the complement graph of $\mathbb{AG}(\mathbb{Z}_n)$. If possible, let $C': \langle x_1\rangle \sim \langle x_2 \rangle \sim \cdots \sim \langle x_t \rangle \sim \langle x_1\rangle $ be an induced odd cycle $C$ of length $t\geq 5$ in $\mathbb{AG}^c(\mathbb{Z}_n)$. As $C'$ consists of $t\geq 5$ vertices, at least one of the vertices must be of 1st or 2nd type.

{\bf Case 1:} $\langle x_1 \rangle$ is a vertex of 1st type, i.e., without loss of generality, let $x_1=p_1$. Now, as $\langle p_1 \rangle$ is a pendant vertex in $\mathbb{AG}(\mathbb{Z}_n)$, $\langle p_1 \rangle$ is not adjacent to exactly one vertex in $\mathbb{AG}^c(\mathbb{Z}_n)$. Thus $C'$ always contain a chord, a contradiction.

{\bf Case 2:} $\langle x_1 \rangle$ is a vertex of 2nd type, i.e., without loss of generality, let $x_1=p_1p_2$. As degree of $\langle p_1p_2 \rangle$ in $\mathbb{AG}(\mathbb{Z}_n)$ is $3$, the number of vertices which are not adjacent to $\langle p_1p_2 \rangle$ in $\mathbb{AG}^c(\mathbb{Z}_n)$ is $3$. Thus, as $C'$ is chordless, it must be an induced $5$-cycle, i.e., $$C':\langle p_1p_2 \rangle \sim \langle x_2 \rangle \sim \langle x_3\rangle \sim \langle x_4 \rangle \sim \langle x_5 \rangle \sim \langle p_1p_2 \rangle$$
As $\langle x_3\rangle, \langle x_4 \rangle$ are adjacent to $\langle p_1p_2 \rangle$ in $\mathbb{AG}(\mathbb{Z}_n)$, we must have $x_3,x_4 \in \{p_3p_4,p_1p_3p_4,p_2p_3p_4 \}$. If $\{x_3,x_4 \}=\{p_1p_3p_4,p_2p_3p_4 \}$, then $\langle x_3\rangle \sim \langle x_4 \rangle$ in $\mathbb{AG}(\mathbb{Z}_n)$. Thus, without loss of generality, we can assume $x_3=p_3p_4$ and $x_4=p_1p_3p_4$, i.e.,  
$$C':\langle p_1p_2 \rangle \sim \langle x_2 \rangle \sim \langle p_3p_4\rangle \sim \langle p_1p_3p_4 \rangle \sim \langle x_5 \rangle \sim \langle p_1p_2 \rangle$$
As $\langle x_5 \rangle \not\sim \langle p_3p_4\rangle$ in $\mathbb{AG}^c(\mathbb{Z}_n)$, we have $p_1p_2\mid x_5$. Thus $x_5=p_1p_2p_3$ or $p_1p_2p_4$. However, in any case, $\langle x_5 \rangle \sim \langle p_1p_3p_4 \rangle$ in $\mathbb{AG}(\mathbb{Z}_n)$, a contradiction.

Thus $\mathbb{AG}^c(\mathbb{Z}_n)$ has no induced odd cycle $C$ of length $t\geq 5$. Hence, by strong perfect graph theorem, the lemma follows.
\end{proof}

Now, we turn towards the case when $n$ has exactly three distinct prime factors and characterize the perfect graphs among this subfamily.

\begin{lemma}
	If $n={p_1}^{\alpha_1}{p_2}^{\alpha_2}{p_3}^{\alpha_3}$ and $\alpha_{i}>1$ for at least two $i$'s, then $\mathbb{AG}(\mathbb{Z}_n)$ is not perfect.	
\end{lemma}
\begin{proof}
Let $\alpha_{i}\geq 2$ for at least two $i$'s, say $\alpha_1,\alpha_2\geq 2$. Then the following five vertices, taken in order, $$\langle {p_2}^{\alpha_2}{p_3}^{\alpha_3} \rangle, \langle {p_1}^{\alpha_1}{p_2} \rangle, \langle {p_1}{p_2}^{\alpha_2-1}{p_3}^{\alpha_3} \rangle, \langle {p_1}^{\alpha_1-1}{p_2}^{\alpha_2}\rangle, \langle {p_1}^{\alpha_1}{p_3}^{\alpha_3} \rangle$$ form an induced $5$-cycle in $\mathbb{AG}(\mathbb{Z}_n)$. As earlier, the adjacency and non-adjacency follows from Proposition \ref{adj}. Hence, by strong perfect graph theorem, $\mathbb{AG}(\mathbb{Z}_n)$ is not perfect.
\end{proof}
So, now we assume that $n={p_1}^{\alpha_1}{p_2}{p_3}$.

\begin{lemma}\label{3-factor-no-odd-cycle}
	If $n={p}^{\alpha}{q}r$, then $\mathbb{AG}(\mathbb{Z}_n)$ has no induced odd cycle of length greater than $3$.
\end{lemma}
\begin{proof}
	If possible, let $\mathbb{AG}(\mathbb{Z}_n)$ has an induced odd cycle $C: \langle x_1\rangle \sim \langle x_2 \rangle \sim \cdots \sim \langle x_t \rangle \sim \langle x_1\rangle $, where $x_i=p^{\alpha_{i}}q^{\beta_i}r^{\gamma_i}$ for $i=1,2,\ldots,t$.
	
	{\it Claim 1:} For all $i\in \{1,2,\ldots,t\}$, either $\alpha_{i}> \alpha/2$ or one of $\beta_{i}, \gamma_i \neq 0$.
	
	{\it Proof of Claim 1:} If possible let $\alpha_i \leq \alpha/2$ and $\beta_i =\gamma_i=0$. Now  $\langle x_i\rangle \sim \langle x_{i+1}\rangle$ and $\langle x_i\rangle \sim \langle x_{i-1}\rangle$ imply $\alpha_{i+1}, \alpha_{i-1} \geq \alpha/2$ and $\beta_{i+1}=\beta_{i-1}=1=\gamma_{i+1}=\gamma_{i-1}$. Hence we have $\langle x_{i+1}\rangle \sim \langle x_{i-1}\rangle$, which is a contradiction.

	{\it Claim 2:} For all $i\in \{1,2,\ldots,t\}$, either $\alpha_{i}< \alpha/2$ or one of $\beta_{i}, \gamma_i \neq 1$.
	
	{\it Proof of Claim 2:} If possible let $\alpha_i \geq \alpha/2$ and $\beta_i =\gamma_i=1$. Now  $\langle x_i\rangle \nsim \langle x_{i+2}\rangle$ and $\langle x_i\rangle \nsim \langle x_{i+3}\rangle$ imply $\alpha_{i+2}, \alpha_{i+3} < \alpha/2$, hence $\alpha_{i+2}+\alpha_{i+3}< \alpha$, which is a contradiction as $\langle x_{i+2}\rangle \sim \langle x_{i+3}\rangle$. 
	
	{\it Claim 3:} For all $i\in \{1,2,\ldots,t\}$, either $\alpha_{i}> \alpha/2$ or one of $\beta_{i}, \gamma_i \neq 1$.
	
	{\it Proof of Claim 3:} Without loss of generality let $\alpha_1 \leq \alpha/2$ and $\beta_1 =\gamma_1=1$. Now $\langle x_1\rangle \sim \langle x_{2}\rangle$ and $\langle x_1\rangle \sim \langle x_{t}\rangle$ imply $\alpha_{2}, \alpha_{t} \geq \alpha/2$. As $\langle x_2\rangle \nsim \langle x_{t}\rangle$ then either $\beta_2 + \beta_t=0$ or $\gamma_2+ \gamma_t=0$ or both. Again without loss of generality we can take $\beta_2+ \beta_t=0$, i.e., $\beta_2= \beta_t=0$ and hence $\langle x_2\rangle \sim \langle x_{3}\rangle$ and $\langle x_t\rangle \sim \langle x_{t-1}\rangle$ imply $\beta_3=1=\beta_{t-1}$. Now $\langle x_1\rangle \sim \langle x_{t}\rangle$ and $\langle x_1\rangle \nsim \langle x_{4}\rangle$ imply $\alpha_1+\alpha_t \geq \alpha$ and $\alpha_1+ \alpha_4 < \alpha$. From these two equations we have $\alpha_t > \alpha_4$. Therefore $\alpha_3+ \alpha_4 \geq \alpha$ imply $\alpha_3+\alpha_t > \alpha$. So $\langle x_t\rangle \nsim \langle x_{3}\rangle$ and  $\beta_3=1$  imply $\gamma_t +\gamma_3=0$, i.e., $\gamma_3=\gamma_t=0$. Therefore $\gamma_{t-1}=1$. As  $\langle x_2\rangle \nsim \langle x_{t-1}\rangle$ and $\beta_{t-1}=\gamma_{t-1}=1$ hence $\alpha_2 + \alpha_{t-1}< \alpha$ and we know $\alpha_1 +\alpha_2 \geq \alpha$. From these two equations we have $\alpha_1 > \alpha_{t-1}$. So $\alpha_{t-1}+ \alpha_{t-2} \geq \alpha$ imply $\alpha_1 + \alpha_{t-2} > \alpha$ and $\beta_1=\gamma_1=1$, so we have $\langle x_1\rangle \sim \langle x_{t-2}\rangle$, which is a contradiction.
	
	From Claim 2 and Claim 3 we see that for any $i$, both $\alpha_i$ and $\beta_i$ can not be $1$ and hence both can not be $0$. 
	
	{\it Claim 4:} For all $i\in \{1,2,\ldots,t\}$, $\alpha_{i}> \alpha/2$ .
	
	{\it Proof of Claim 4:} Without loss of generality let $\alpha_1 \leq \alpha/2$ and $\beta_1 =1, \gamma_1=0$. Now $\langle x_1\rangle \sim \langle x_{2}\rangle$ and $\langle x_1\rangle \sim \langle x_{t}\rangle$ imply $\alpha_{2}, \alpha_{t} \geq \alpha/2$ and $\gamma_2= \gamma_t=1$. As $\langle x_2\rangle \nsim \langle x_{t}\rangle$, $\alpha_2+ \alpha_t \geq \alpha$ and $\gamma_2+\gamma_t=2$, we have $\beta_2+\beta_t=0$, i.e., $\beta_2=\beta_t=0$. Hence $\beta_3=\beta_{t-1}=1$. Now $\langle x_3\rangle \nsim \langle x_{t}\rangle$ and $\beta_3=1=\gamma_t$ imply $\alpha_3+\alpha_t< \alpha$ and $\alpha_2+\alpha_3 \geq \alpha$, hence $\alpha_2 > \alpha_t$. Therefore $\alpha_t+\alpha_{t-1} \geq \alpha$ implies $\alpha_2+\alpha_{t-1} > \alpha$. So $\beta_{t-1}=\gamma_2=1$ imply $\langle x_2\rangle \sim \langle x_{t-1}\rangle$,  which is impossible and hence the Claim holds. 
	
	So from the Claim 1 and Claim 4 we can consider $\alpha_1> \alpha/2$,  $\beta_1=1$ and $\gamma_1=0$. From Claim 4 we have $\alpha_3, \alpha_4 > \alpha/2$. So $\langle x_1\rangle \nsim \langle x_{3}\rangle$, $\alpha_1+\alpha_3 > \alpha$ and $\beta_1=1$ imply $\gamma_1+\gamma_3=0$, i.e., $\gamma_3=0$, i.e., $\gamma_4=1$. Therefore $\alpha_1 +\alpha_4 > \alpha$ and $\beta_1=\gamma_4=1$ imply $\langle x_1\rangle \sim \langle x_{4}\rangle$, which is a contradiction. This completes the proof.
\end{proof}

\begin{lemma}\label{3-factor-no-odd-cycle-complement}
	If $n={p}^{\alpha}{q}r$, then $\mathbb{AG}^c(\mathbb{Z}_n)$ has no induced odd cycle of length greater than $3$.
\end{lemma}
\begin{proof}
	We start by noting that $\langle a \rangle \sim \langle b \rangle$ in $\mathbb{AG}^c(\mathbb{Z}_n)$ if and only if $n\nmid ab$. If possible, let $\mathbb{AG}^c(\mathbb{Z}_n)$ has an induced odd cycle $C: \langle x_1\rangle \sim \langle x_2 \rangle \sim \cdots \sim \langle x_t \rangle \sim \langle x_1\rangle $, where $x_i=p^{\alpha_{i}}q^{\beta_i}r^{\gamma_i}$ for $i=1,2,\ldots,t$.
	
	{\it Claim 1:}  For all $i\in \{1,2,\ldots,t\}$, either $\alpha_{i}> \alpha/2$ or one of $\beta_{i}, \gamma_i \neq 0$.
	
	{\it Proof of Claim 1:} If possible let $\alpha_i \leq \alpha/2$ and $\beta_i =\gamma_i=0$. Now $ \langle x_{i+2}\rangle$, $\langle x_{i+3}\rangle \nsim \langle x_i \rangle$ imply $\alpha_{i+2}, \alpha_{i+3} \geq\alpha/2$, $\beta_{i+2}=\beta_{i+3}=1=\gamma_{i+2}=\gamma_{i+3}$. Therefore from this we have $\langle x_{i+2}\rangle \nsim \langle x_{i+3}\rangle$, which is a contradiction.
	
	{\it Claim 2:} For all $i\in \{1,2,\ldots,t\}$, either $\alpha_{i}< \alpha/2$ or one of $\beta_{i}, \gamma_i \neq 1$.
	
	{\it Proof of Claim 2:} If possible let $\alpha_i \geq \alpha/2$ and $\beta_i =\gamma_i=1$.  As $ \langle x_{i-1}\rangle$ and $\langle x_{i+1}\rangle \sim \langle x_i \rangle$, hence $\alpha_{i-1}, \alpha_{i+1} <\alpha/2$, i.e., $\alpha_{i-1}+ \alpha_{i+1}<\alpha$ and hence $\langle x_{i-1}\rangle \sim \langle x_{i+1} \rangle$, which is a contradiction.
	
	{\it Claim 3:} For all $i\in \{1,2,\ldots,t\}$, either $\alpha_{i}> \alpha/2$ or one of $\beta_{i}, \gamma_i \neq 1$.
	
	{\it Proof of Claim 3:} Without loss of generality let $\alpha_1 \leq \alpha/2$ and $\beta_1 =\gamma_1=1$.  As $ \langle x_{3}\rangle$, $\langle x_{4}\rangle \nsim \langle x_1 \rangle$ hence $\alpha_3, \alpha_4 \geq \alpha/2$, i.e., $\alpha_3 + \alpha_4 \geq \alpha$. Now $\langle x_{3}\rangle \sim \langle x_4 \rangle$ implies either $\beta_3 +\beta_4 =0$ or $\gamma_3 + \gamma_4=0$ or both. Without loss of generality we can assume $\beta_3 + \beta_4=0$, i.e., $\beta_3=0=\beta_4$. Now $\langle x_{2}\rangle \nsim \langle x_4 \rangle$,  $\beta_4=0$ imply $\beta_2=1$ and $\langle x_{3}\rangle \nsim \langle x_t \rangle$,  $\beta_3=0$ imply $\beta_t=1$. Again $\langle x_{1}\rangle \sim \langle x_2 \rangle$ and $\beta_1=1=\gamma_1$  imply $\alpha_1 +\alpha_2 < \alpha$. Therefore $\alpha_2 + \alpha_t \geq \alpha$ implies $\alpha_t > \alpha_1$. So $\alpha_1 + \alpha_{t-1} \geq \alpha$ imply $\alpha_t + \alpha_{t-1} > \alpha$. Now $\langle x_{t}\rangle \sim \langle x_{t-1} \rangle$ and $\beta_t=1$ imply $\gamma_t + \gamma_{t-1}=0$, i.e., $\gamma_t=0$. Again $\langle x_{1}\rangle \sim \langle x_t \rangle$ and $\beta_1=1=\gamma_1$  imply $\alpha_1 +\alpha_t < \alpha$. So $\alpha_2 + \alpha_t \geq \alpha$ imply $\alpha_2 > \alpha_1$. So $\alpha_1 + \alpha_3 \geq \alpha$ imply $\alpha_2 + \alpha_3 > \alpha$. Now $\langle x_{2}\rangle \sim \langle x_{3} \rangle$ and $\beta_2=1$ imply $\gamma_2 + \gamma_{3}=0$, i.e., $\gamma_2=0$. Therefore $\gamma_2= 0= \gamma_t$ implies $\langle x_{2}\rangle \sim \langle x_{t} \rangle$, which is impossible.
	
	From Claim 2 and Claim 3 we see that for any $i$, both $\alpha_i$ and $\beta_i$ can not be $1$ and hence both can not be $0$. 
	
	{\it Claim 4:} For all $i\in \{1,2,\ldots,t\}$, $\alpha_{i}> \alpha/2$ .
	
	{\it Proof of Claim 4:} Without loss of generality let $\alpha_1 \leq \alpha/2$ and $\beta_1 =1, \gamma_1=0$. Therefore $ \langle x_{3}\rangle$, $\langle x_{4}\rangle \nsim \langle x_1 \rangle$ imply $\alpha_3, \alpha_4 \geq \alpha/2$ and $\gamma_3 =1 =\gamma_4$. So $\langle x_{3}\rangle \sim \langle x_4 \rangle$ imply $\beta_3 + \beta_4 =0$, i.e., $\beta_3=0 = \beta_4$. Now $\beta_4=0$ and $\langle x_{4}\rangle \nsim \langle x_2 \rangle$  imply $\beta_2=1$. Now $\langle x_{2}\rangle \sim \langle x_3 \rangle$ and $\beta_2=1=\gamma_3$ imply $\alpha_2 + \alpha_3 < \alpha$. So $\alpha_1+ \alpha_3 \geq \alpha$ imply $\alpha_1 > \alpha_2$. Also $\alpha_2 +\alpha_t \geq \alpha$ imply $\alpha_1 + \alpha_t > \alpha$. So $\langle x_{1}\rangle \nsim \langle x_t \rangle$ and $\beta_1=1$ imply $\gamma_1 + \gamma_t=0$, i.e., $\gamma_t=0$, i.e., $\gamma_2=1$.Hence we have $\beta_2=1=\gamma_2$, which is not possible by Claim 2 and Claim 3. 
	
	So from the Claim 1 and Claim 4 we can consider $\alpha_1> \alpha/2$,  $\beta_1=1$ and $\gamma_1=0$. So $\langle x_{3}\rangle , \langle x_4 \rangle  \nsim \langle x_1 \rangle$ imply $\gamma_3=1=\gamma_4$.  From Claim 4 we have $\alpha_3, \alpha_4 > \alpha/2$. So $\langle x_3 \rangle  \sim \langle x_4 \rangle$ implies $\beta_3 +\beta_4=0$, i.e., $\beta_3=0=\beta_4$. Now $\langle x_2 \rangle  \sim \langle x_3 \rangle$, $\gamma_{3}=1$ and $\alpha_2, \alpha_3 > \alpha/2$ imply $\beta_2+ \beta_3=0$, i.e., $\beta_2=0$. Therefore  $\beta_2=0=\beta_4$ imply $\langle x_2 \rangle  \sim \langle x_4 \rangle$, which is a contradiction and this completes the proof.
\end{proof}

Thus, it follows from strong perfect graph theorem and Lemma \ref{3-factor-no-odd-cycle} and Lemma \ref{3-factor-no-odd-cycle-complement}, that if $n=p^\alpha qr$, then $\mathbb{AG}(\mathbb{Z}_n)$ is perfect.

Thus, the case when $n$ has three distinct prime factors is complete. Now, we focus on the case, when $n$ has two distinct prime factors.

\begin{lemma}\label{2-factor-no-odd-cycle}
	If $n={p}^{\alpha}{q}^{\beta}$, then $\mathbb{AG}(\mathbb{Z}_n)$ has no induced odd cycle of length greater than $3$.	
\end{lemma}
\begin{proof}
If possible, let $\mathbb{AG}(\mathbb{Z}_n)$ has an induced odd cycle $C: \langle x_1\rangle \sim \langle x_2 \rangle \sim \cdots \sim \langle x_t \rangle \sim \langle x_1\rangle $, where $x_i=p^{\alpha_{i}}q^{\beta_i}$ for $i=1,2,\ldots,t$.

{\it Claim 1:} For all $i\in \{1,2,\ldots,t\}$, either $\alpha_{i}> \alpha/2$ or $\beta_{i}> \beta/2$.

{\it Proof of Claim 1:} If $\alpha_{i}\leq \alpha/2$ and $\beta_{i}\leq \beta/2$ for some $i$, then as $\langle x_i\rangle \sim \langle x_{i+1}\rangle$, we have $\alpha_{i+1}\geq \alpha/2$ and $\beta_{i+1}\geq \beta/2$. Similarly, as $\langle x_i\rangle \sim \langle x_{i-1}\rangle$, we have $\alpha_{i-1}\geq \alpha/2$ and $\beta_{i-1}\geq \beta/2$. But this implies $\alpha_{i+1}+\alpha_{i-1}\geq \alpha$ and $\beta_{i+1}+\beta_{i-1}\geq \beta$, i.e., $\langle x_{i-1}\rangle \sim \langle x_{i+1}\rangle$, a contradiction. Thus the claim holds. 

In Claim 1, we show that for any $i$, either $\alpha_{i}$ or $\beta_{i}$ is greater than $\alpha/2$ or $\beta/2$ respectively. In the next claim, we show that both of them can not be greater or equal to $\alpha/2$ and $\beta/2$ simultaneously.

{\it Claim 2:} For any $i$, both $\alpha_i\geq \alpha/2$ and $\beta_i\geq \beta/2$ can not hold.

{\it Proof of Claim 2:} Without loss of generality, suppose $\alpha_1\geq \alpha/2$ and $\beta_1\geq \beta/2$. As $\langle x_1\rangle \not\sim \langle x_3\rangle$, we have either $\alpha_1+\alpha_3<\alpha$ or $\beta_1+\beta_3<\beta$, i.e., $\alpha_3<\alpha/2$ or $\beta_3<\beta/2$. Again, without loss of generality, we assume that $\alpha_3<\alpha/2$. So, by Claim 1, we get $\beta_3>\beta/2$. As $\langle x_3 \rangle$ is adjacent to both $\langle x_2\rangle$ and $\langle x_4 \rangle$, we have $\alpha_2,\alpha_4>\alpha/2$. As $\langle x_1 \rangle \not\sim \langle x_4 \rangle$ and $\alpha_1,\alpha_4\geq \alpha/2$ and $\beta_1\geq \beta/2$, we have $\beta_4<\beta/2$. Again as $\langle x_4 \rangle \sim \langle x_5 \rangle$, we have $\beta_5>\beta/2$.

Here $C$ is a $t$-cycle with $t$ odd and $t\geq 5$. We show, by strong induction, that for any odd value of $t\geq 5$, we get a contradiction. 

We start with $t=5$, i.e., $\langle x_1 \rangle \sim \langle x_5 \rangle$. As $\langle x_1 \rangle \not\sim \langle x_4 \rangle$ and $\alpha_1,\alpha_4\geq \alpha/2$, we have $\beta_1+\beta_4<\beta$. Again as $\langle x_4 \rangle \sim \langle x_5 \rangle$, we have $\beta_4+\beta_5\geq \beta$. Thus, we get $\beta_5>\beta_1$. As $\langle x_1 \rangle \sim \langle x_2 \rangle$, we have $\beta_1+\beta_2\geq \beta$, i.e., $\beta_2+\beta_5>\beta$. Thus as $\langle x_2 \rangle \not\sim \langle x_5 \rangle$, we must have $\alpha_2+\alpha_5<\alpha$. Also, as $\langle x_1 \rangle \sim \langle x_5 \rangle$, we have $\alpha_1+\alpha_5\geq \alpha$. Thus we must have $\alpha_1>\alpha_2$. Similarly, $\langle x_2 \rangle \sim \langle x_3 \rangle$ implies $\alpha_2+\alpha_3\geq \alpha$, i.e., $\alpha_1+\alpha_3>\alpha$. On the other hand, as $\beta_1,\beta_3\geq \beta/2$, we have $\beta_1+\beta_3\geq \beta$. Thus we have $\langle x_1 \rangle \sim \langle x_3 \rangle$. Hence we get a contradiction for $t=5$.

For $t>5$, as $\langle x_1 \rangle \not\sim \langle x_5 \rangle$ and $\beta_1,\beta_5\geq \beta/2$ and $\alpha_1\geq \alpha/2$, we have $\alpha_5<\alpha/2$. Thus the induction hypothesis is: For all odd $k$ satisfying $1<k<t-2$, 
$$\begin{array}{ccc}
\alpha_i<\alpha/2, ~1<i\leq k, ~i \mbox{ is odd} & \mbox{and }& \beta_{i}\geq \beta/2, ~1\leq i \leq k, ~i \mbox{ is odd}\\
\alpha_j\geq \alpha/2, ~2<j\leq k-1, ~j \mbox{ is even} & \mbox{and }& \beta_{j}< \beta/2, ~2< j \leq k-1, ~j \mbox{ is even}\\
\end{array}$$

Now $\langle x_{k+1}\rangle\sim \langle x_k\rangle$ and $\alpha_k <\alpha/2$ imply $\alpha_{k+1}>\alpha/2$. Similarly,  $\langle x_1\rangle \not\sim \langle x_{k+1}\rangle$ and $\beta_1 \geq \beta/2 $  implies $\beta_{k+1}<\beta/2$ and $\langle x_{k+1}\rangle \sim \langle x_{k+2}\rangle$ implies $\beta_{k+2}>\beta/2$. As $k+2\leq t-2$, we have $\langle x_1\rangle \not\sim \langle x_{k+2}\rangle$ and $\beta_1,\beta_{k+2}>\beta/2$, which implies $\alpha_{k+2}<\alpha/2$. Thus, by induction, we have 
$$\begin{array}{c}
\mbox{For odd }i \mbox{ with }1<i<t, \alpha_{i}< \alpha/2 \mbox{ and, for odd }i \mbox{ with }1\leq i\leq t, \beta_i\geq \beta/2\\
\mbox{For even }j \mbox{ with }j>2, \alpha_{j}\geq \alpha/2 \mbox{ and } \beta_{j}<\beta/2.
\end{array}$$
Now, $\langle x_2 \rangle \not\sim \langle x_t \rangle$ implies either $\beta_2+\beta_t<\beta$ or $\alpha_2+\alpha_t< \alpha$ or both. As $t$ is odd, $t-1$ is even and hence $\alpha_1,\alpha_{t-1}\geq\alpha/2$ and $\beta_1>\beta/2$. Thus $\langle x_1 \rangle \not\sim \langle x_{t-1} \rangle$ implies $\beta_1+\beta_{t-1}<\beta$ and $\langle x_t \rangle\sim \langle x_{t-1} \rangle$ implies $\beta_t+\beta_{t-1}\geq \beta$. Therefore $\beta_t>\beta_1$.

Again $\langle x_1 \rangle \sim \langle x_{2} \rangle$ implies $\beta_1+\beta_2\geq \beta$, i.e., $\beta_t+\beta_2>\beta$. Now, as $\langle x_2 \rangle \not\sim \langle x_{t} \rangle$, we must have $\alpha_2+\alpha_t<\alpha$.

Also $\langle x_1 \rangle \sim \langle x_{t} \rangle$ implies $\alpha_1+\alpha_t\geq \alpha$. Therefore $\alpha_1>\alpha_2$. Similarly $\langle x_2 \rangle \sim \langle x_{3} \rangle$ implies $\alpha_2+\alpha_3\geq \alpha$. Thus $\alpha_1+\alpha_3> \alpha$. Again, as $\beta_1,\beta_3>\beta/2$, we have $\langle x_1 \rangle \sim \langle x_{3} \rangle$, a contradiction. Hence Claim 2 holds good.

From Claim 1 and 2, we see that for any $i$, both $\alpha_i,\beta_i$ can not be simultaneously `greater or equal' or `lesser or equal' to $\alpha/2$ and $\beta/2$ respectively. So for any $i$, either $\alpha_i<\alpha/2,\beta_i > \beta/2$ or $\alpha_i> \alpha/2,\beta_i< \beta/2$ holds. Without loss of generality, let $\alpha_1<\alpha/2,\beta_1> \beta/2$.

Now, as $\langle x_1 \rangle \sim \langle x_{2} \rangle$, we have $\alpha_1+\alpha_2\geq \alpha$, which implies $\alpha_2>\alpha/2$, i.e., $\beta_2<\beta/2$ (by Claim 2). Similarly $\langle x_2 \rangle \sim \langle x_{3} \rangle$ implies $\beta_3>\beta/2$, i.e., $\alpha_3<\alpha/2$ (by Claim 2). Proceeding this way, we get 
$$\begin{array}{c}
\mbox{If }i\mbox{ is odd, }\alpha_i<\alpha/2 \mbox{ and }\beta_i> \beta/2\\
\mbox{If }i\mbox{ is even, }\alpha_i> \alpha/2\mbox{ and }\beta_i< \beta/2
\end{array}$$
As $t$ is odd, we have $\alpha_t<\alpha/2$. Also, as $\langle x_1 \rangle \sim \langle x_{t} \rangle$, we have $\alpha_1+\alpha_t\geq \alpha$. However as $\alpha_1,\alpha_t<\alpha/2$, we get a contradiction. Thus $\mathbb{AG}(\mathbb{Z}_n)$ has no induced odd cycle of length greater than $3$.
\end{proof}

\begin{lemma}\label{2-factor-no-odd-cycle-complement}
	If $n={p}^{\alpha}{q}^{\beta}$, then $\mathbb{AG}^c(\mathbb{Z}_n)$ has no induced odd cycle of length greater than $3$.
\end{lemma}
\begin{proof}
We start by noting that $\langle a \rangle \sim \langle b \rangle$ in $\mathbb{AG}^c(\mathbb{Z}_n)$ if and only if $n\nmid ab$. If possible, let $\mathbb{AG}^c(\mathbb{Z}_n)$ has an induced odd cycle $C: \langle x_1\rangle \sim \langle x_2 \rangle \sim \cdots \sim \langle x_t \rangle \sim \langle x_1\rangle $, where $x_i=p^{\alpha_{i}}q^{\beta_i}$ for $i=1,2,\ldots,t$.

{\it Claim 1:} For all $i\in \{1,2,\ldots,t\}$, either $\alpha_{i}> \alpha/2$ or $\beta_{i}> \beta/2$.

{\it Proof of Claim 1:} If $\alpha_{i}\leq \alpha/2$ and $\beta_{i}\leq \beta/2$ for some $i$, then as $\langle x_i\rangle \not\sim \langle x_{i+2}\rangle$ and $\langle x_i\rangle \not\sim \langle x_{i+3}\rangle$, we have $\alpha_{i+2},\alpha_{i+3}\geq \alpha/2$ and $\beta_{i+2},\beta_{i+3}\geq \beta/2$. But this imply that $\langle x_{i+2}\rangle \not\sim \langle x_{i+3}\rangle$ in $\mathbb{AG}^c(\mathbb{Z}_n)$, a contradiction. Hence Claim 1 holds.
In Claim 1, we show that for any $i$, either $\alpha_{i}$ or $\beta_{i}$ is greater than $\alpha/2$ or $\beta/2$ respectively. In the next claim, we show that both of them can not be greater or equal to $\alpha/2$ and $\beta/2$ simulatneously.

{\it Claim 2:} For any $i$, both $\alpha_i\geq \alpha/2$ and $\beta_i\geq \beta/2$ can not hold.

{\it Proof of Claim 2:} Without loss of generality, suppose $\alpha_1\geq \alpha/2$ and $\beta_1\geq \beta/2$. $\langle x_1\rangle \sim \langle x_2\rangle $ implies either $\alpha_1 +\alpha_2 < \alpha$ or $\beta_1+ \beta_2 <\beta$ or both.  Again, without loss of generality, we assume that $\alpha_1 +\alpha_2 < \alpha$, i.e., $\alpha_2 < \alpha/2$. Now $\langle x_2\rangle \nsim \langle x_t\rangle$ implies $\alpha_2 + \alpha_t \geq \alpha$, i.e., $\alpha_t > \alpha/2$. 

At first we assume that $t=5$. Therefore $\langle x_1\rangle \sim \langle x_5\rangle$ and $\alpha_1 ,\alpha_5 \geq \alpha/2$ imply $\beta_1 + \beta_5 < \beta$. Now $\langle x_1\rangle \nsim \langle x_3\rangle$ imply $\alpha_1+ \alpha_3 \geq \alpha$, so $\alpha_1 +\alpha_2 < \alpha$ implies $\alpha_3 > \alpha_2$. Again $\langle x_3\rangle \sim \langle x_4\rangle$ imply either $\alpha_3 +\alpha_4 < \alpha$ or $\beta_3+ \beta_4 < \beta$ or both. 
Now $\langle x_2\rangle \nsim \langle x_4\rangle$ imply $\alpha_2 + \alpha_4 \geq \alpha$. If $\alpha_3 + \alpha_4< \alpha$, then we have $\alpha_2 > \alpha_3$, which is a contradiction as we already have $\alpha_3> \alpha_2$.
 Now $\langle x_1\rangle \nsim \langle x_4\rangle$ implies $\beta_1+\beta_4 \geq \beta$. If $\beta_3+ \beta_4 < \beta$, then we have $\beta_1 > \beta_3$. Now $\langle x_3\rangle \nsim \langle x_5\rangle$ implies $\beta_3 +\beta_5 \geq \beta$, therefore $\beta_1 + \beta_5 > \beta$, which contradicts the condition $\beta_1+ \beta_5 < \beta$. So for $t=5$ the Claim 2 is true.

Now assume that $t>5$. As $\alpha_1, \alpha_t \geq \alpha/2$ and $\langle x_1\rangle \sim \langle x_t\rangle$, we have $\beta_1 + \beta_t <\beta$, i.e., $\beta_t <\beta/2$ as $\beta_1 \geq \beta/2$. Now $\alpha_2 < \alpha/2$ and $\langle x_4\rangle ,\langle x_5\rangle  \nsim \langle x_2\rangle $ imply $\alpha_4, \alpha_5 > \alpha/2$, hence $\alpha_4+ \alpha_5 > \alpha$.
Again $\beta_t <\beta/2$ and $\langle x_4\rangle,\langle x_5\rangle \nsim \langle x_t\rangle$ imply $\beta_4, \beta_5 > \beta/2$, hence $\beta_4+\beta_5> \beta$, which is a contradiction as $\langle x_4\rangle \sim \langle x_5\rangle$. Hence Claim 2 holds for all odd $t \geq 5$.

From Claim 1 and 2, we see that for any $i$, both $\alpha_i,\beta_i$ can not be simultaneously `greater or equal' or `lesser or equal' to $\alpha/2$ and $\beta/2$ respectively. So for any $i$, either $\alpha_i<\alpha/2,\beta_i>\beta/2$ or $\alpha_i > \alpha/2,\beta_i< \beta/2$ holds. Without loss of generality, let $\alpha_1<\alpha/2,\beta_1 > \beta/2$.

Now $\langle x_3\rangle, \langle x_4\rangle \nsim \langle x_1\rangle$ imply $\alpha_3, \alpha_4 >\alpha/2$ and hence by Claim 2 we have $\beta_3, \beta_4 < \beta/2$. As $\langle x_2\rangle \nsim \langle x_4\rangle$, so $\beta_2 >\beta/2$ and by Claim 2 we have $\alpha_2< \alpha/2$. Now $\langle x_2\rangle \nsim \langle x_5\rangle$ implies $\alpha_5 >\alpha/2$. Then by Claim 2 we have $\beta_5< \beta/2$, but $\beta_3< \beta/2$ imply $\beta_3+ \beta_5 <\beta$, which is a contradiction as $\langle x_3\rangle \nsim \langle x_5\rangle$. Thus $\mathbb{AG}^c(\mathbb{Z}_n)$ has no induced odd cycle of length greater than $3$.
\end{proof}

Thus from Lemma \ref{2-factor-no-odd-cycle} and Lemma \ref{2-factor-no-odd-cycle-complement}, we have if $n={p}^{\alpha}{q}^{\beta}$, then $\mathbb{AG}(\mathbb{Z}_n)$ is perfect.	

Now, we deal with the last case when $n$ is a prime power.
\begin{lemma}
	If $n={p}^{\alpha}$, then $\mathbb{AG}(\mathbb{Z}_n)$ is perfect.	
\end{lemma}
\begin{proof}
In this case, the vertices are $\langle p \rangle, \langle p^2 \rangle,\ldots, \langle {p}^{\alpha-1} \rangle$ and two vertices $\langle {p}^{k} \rangle$ and $\langle {p}^{l} \rangle$ are adjacent if and only if $k+l\geq \alpha$.

If possible, let $C: \langle {p}^{k_1} \rangle \sim \langle {p}^{k_2} \rangle \sim \cdots \sim \langle {p}^{k_t} \rangle \sim \langle {p}^{k_1} \rangle$ be an induced odd cycle of length $t\geq 5$. Then from adjacency and non-adjacency conditions, we have the following two sets of relations. Adding them, we get a contradiction:
$$\begin{array}{ccc}
k_1+k_2\geq \alpha & ~~~~~~~~~~~~& k_1+k_3<\alpha\\
k_2+k_3\geq \alpha & & k_2+k_4<\alpha\\
\vdots & & \vdots \\
k_{t-1}+k_t\geq \alpha & & k_{t-1}+k_1<\alpha\\
k_t+k_1\geq \alpha & & k_t+k_2<\alpha\\
& & \\
\hline & & \\
2(k_1+k_2+\cdots+k_t)\geq t\alpha & &  2(k_1+k_2+\cdots+k_t)< t\alpha
\end{array}$$
Thus $\mathbb{AG}(\mathbb{Z}_n)$ has no induced odd cycle $C$ of length $t\geq 5$. Proceeding similarly, it can be shown that $\mathbb{AG}^c(\mathbb{Z}_n)$ also has no induced odd cycle of length $t\geq 5$. Hence $\mathbb{AG}(\mathbb{Z}_n)$ is perfect.
\end{proof}

Combining all the results in this section, we get the proof of Theorem \ref{main-theorem}.

\section*{Acknowledgement}
The first and third authors acknowledge the funding of DST-SERB-SRG Sanction no. $SRG/2019/000475$ and $SRG/2019/000684$, Govt. of India. The second author is supported by the PhD fellowship of CSIR (File no. 08/155(0086)/2020-EMR-I), Govt. of India.

\end{document}